\documentclass[12pt,a4paper,english,reqno]{amsart}
\usepackage[latin9]{inputenc}
\usepackage{amsthm}
\usepackage{amssymb}
\usepackage{tikz}
\usepackage{setspace}
\usetikzlibrary{arrows,chains,matrix,positioning,scopes}
\usepackage{babel}
\usepackage{babel}\makeatletter
\tikzset{join/.code=\tikzset{after node path={%
\ifx\tikzchainprevious\pgfutil@empty\else(\tikzchainprevious)%
edge[every join]#1(\tikzchaincurrent)\fi}}}
\makeatother

\tikzset{>=stealth',every on chain/.append style={join},
         every join/.style={->}}

\makeatletter

\pdfpageheight\paperheight
\pdfpagewidth\paperwidth

\numberwithin{equation}{section}
\numberwithin{figure}{section}

\usepackage{amsfonts}

\usepackage{amsthm}     
\usepackage{amscd}      
\usepackage{euscript}
\usepackage[matrix,arrow,curve]{xy}
\usepackage{pb-diagram}
\usepackage{enumitem}
\usepackage{hyperref}
\usepackage{amsmath}

{
       \newtheorem{theorem}{Theorem}[section]
       \newtheorem{proposition}[theorem]{Proposition}
       \newtheorem{lemma}[theorem]{Lemma}

{\theoremstyle{definition}

       \newtheorem{definition}[theorem]{Definition}
       
}

\newcommand{\p}{\partial}

\newcommand{\bC}{\mathbb{C}}

\newcommand{\mc}[1]{\mathcal{#1}_{2,3}}
\newcommand{\cp}{\mathbb{CP}^1}
\newcommand{\Ms}{\mathcal{M}_{2,3}^{\text{small}}}
\newcommand{\Mb}{\mathcal{M}_{2,3}^{\text{big}}}
\newcommand{\norm}[1]{\left\lVert#1\right\rVert}

\makeatother

\begin{document}
\begin{onehalfspacing}

\title{Regulated $L^2$ Norm on Certain Wild Higgs Bundles Over $\mathbb{CP}^1$}
\author{Hsiao-Tzu Tsai}\address{Department of Mathematics, National Cheng-Kung University, Tainan, Taiwan\\l18121014@gs.ncku.edu.tw}
\author{Chih-Chung Liu}\address{Department of Mathematics, National Cheng-Kung University, Tainan, Taiwan\\cliu@mail.ncku.edu.tw}
\maketitle

\begin{abstract}
  We define and analyse certain $L^2$ norm on moduli space of Higgs bundles over $\mathbb{CP}^1$ with a particular type of singularity. We prove that certain limit of our metrics here is the regulated $L^2$ norm defined on the central fiber, first appearing in \cite{FN}.
\end{abstract}

\section{introduction}
The moduli space $\mathcal{M}$ of rank $k$ stable Higgs bundles, or equivalently the Hitchin moduli spaces, over a compact Riemann surface $C$ with genus $g$, has been studied extensively from various points of view. On the Higgs bundle side, we have the space consists of {\em Higgs pairs} $(\mathcal{E},\varphi)$, where $\mathcal{E}$ is a rank $k$ holomorphic bundle with gauge group $\mathcal{G}$ and $\varphi$ is a holomorphic section of $End \mathcal{E} \; \otimes K_X$. Here, $K_X$ is the holomorphic cotangent bundle of $X$ (or the canonical bundle). The space, via Hitchin-Kobayashi correspondence, corresponds to the space of special metrics known as the {\em harmonic metrics}. Such metrics, $h$, enable the pair $(\mathcal{E},\varphi)$ to satisfy the {\em Hitchin equation}:

\begin{equation}
  F_h + [\varphi,\varphi^{*_h}]=0.
  \label{Hitchin Equation}
\end{equation}

\noindent Here, $F_h$ is the curvature of the Chern connection determined by $h$ and the holomorphic structure of $\mathcal{E}$. This is a gauge-theoretic equation that may be obtained by dimensional reduction of classical Yang-Mills equation to Riemann surfaces.

A well known structure for these spaces are the {\em Hitchin fibrations}, where $\mathcal{M}$ fibers over the vector space $\mathcal{B}$ of invariant polynomials. Such a structure led to the description of Higgs pairs by spectral curves, which may be thought of as the section of eigenvalues of $\varphi$. The Hitchin fiber over a polynomial is then parametrized by certain line bundles over these curves. These fibers are generically toric, such as the full Jacobian varieties for $\mathcal{G}=GL(k,\bC)$ and the Prym varieties for $\mathcal{G}=SL(k,\bC)$. Hitchin has further shown that the fibration is algebraically completely integrable.

Based on these structures, many profound and delicate analytic results for $k=2$ and $\mathcal{G}=SL(2,\bC)$ have appeared in literature in the past decade (cf. \cite{MSWW},\cite{MSWW1}, \cite{DN}, ...etc ). Most of these results assume the genus $g \geq 2$, and that the spectral curve is smooth. The latter is equivalent to the simpleness of zeros of characteristic polynomial. That is, the polynomial can be written as

\begin{equation}
\lambda^2 - z dz^2
\label{char poly smooth}
\end{equation}

\noindent near its zeros.

These two assumptions make the moduli spaces smooth. However, they somewhat limit the possibilities for explicit computations and examples. For these purposes, some singularities are considered to construct interesting examples. In this paper, we study rank 2 the Higgs bundles over $\mathbb{CP}^1$ wit irregular (wild) singularity at $\infty$, introduced in \cite{FN} and \cite{BP}. The Higgs field $\varphi$ we consider have characteristic polynomial behaving as

\begin{equation}
  \lambda^2 - z^K dz^2
  \label{char poly infty}
\end{equation}

\noindent near $\infty$, for some odd integer $K$. The singularity at $\infty$ is accounted for by considering {\em filtered} bundles with {\em parabolic structures}. The moduli space of these pairs is denoted by $\mathcal{M}_{2,N}$. In \cite{FN}, it was shown that the filtered bundles are decomposed into 2 parabolic line bundles, and the moduli space is stratified according to certain partitions of $N$.

In this article, we focus on the case $N=3$, explicitly described in \cite{FN}.  The space $\mc{M}$ was decomposed into disjoint union of two sets, $\mc{M}^{\text{small}}$ and $\mc{M}^{\text{big}}$. They are in correspondence with $\mathbb{C}$ and $\mathbb{C}^2$, respectively. There are two fixed points, one from each stratum, sitting at the {\em central fiber} of the Hichin fibration. The other fibers consist of $\mc{M}^{\text{big}}$ and a point from  $\mc{M}^{\text{small}}$ being the point of infinity. More details from \cite{FN} will be provided in section 2.

 An important construction there includes a $\mathbb{C}^*$ action on $\mc{M}$  that generalized the classical action defined by Hitchin. At the fixed points of this action, the authors defined the {\em regulated $L^2$ norm} in terms of Higgs pairs and their corresponding harmonic metrics with explicit computation for the case $\mc{M}$. Our first result is to extend the norm to the whole of $\mc{M}$.

\begin{theorem}
  The regulated $L^2$ norm $\mu$ can be extended to a norm on the entire $\mc{M}$.
  \label{Extend Metric intro}
\end{theorem}

\noindent A key step to the extension is the construction of harmonic metrics at all points of $\mc{M}$. We constructed the metrics in the way that the equations for them are similar to the ones considered at fixed point (Toda type equations). In fact, the equations are some smooth perturbations of the ones at fixed point. Elliptic estimates and other techniques in PDE's allow us to show

\begin{theorem}
  The extended norm $\mu$ is smooth on $\mc{M}$.
  \label{Smoothness Metric intro}
\end{theorem}

\noindent The extended norm, in particular, defines a moment map for the $U(1)$-action on $\Ms$ and a small part of $\Mb$. In spite of the achieved extension, the definition on $\Mb$ is still somewhat vague. We are therefore unable to fully answer the conjecture posed in \cite{FN} about the precise description on the monotonicities of $\mu$. We leave it as continuing direction toward other generalizations in the future. 

\section{preliminaries}
While the scope of this paper is to analyze a specific example, we nevertheless list some relevant definitions of filtered bundles and parabolic structures from \cite{FN} for the convenience of the readers. It is a standard fact that the structures of holomorphic vector bundles are characterized by the zeros of global sections. Filtered bundles are defined and characterized, in addition, by global sections with poles. The sheaf of these sections are described in terms of {\em filtration} of modules.

\begin{definition}[Filtered Bundles (\cite{FN})]
  Given a compact Riemann surface $C$ and a divisor $D \subset C$ of finite set of points, a {\em filtered rank $K$-bundle} on $(C,D)$ is a locally free $\mathcal{O}_C(*D)$-module $\mathcal{E}$ of rank $K$, with an increasing filtration by locally free $\mathcal{O}_C$-submodules $\{\mathcal{P}_\alpha \mathcal{E}\}_{\alpha \in \mathbb{R}}$ such that
  \begin{itemize}
    \item $\mathcal{P}_\alpha \mathcal{E}_{C - D} = \mathcal{E}_{C-D}.$
    \item $\mathcal{P}_\alpha \mathcal{E} = \cap_{\beta > \alpha} \mathcal{P}_\beta \mathcal{E}.$
    \item If $x$ is a local coordinate on an open set $U$ around $p \in D$, then 
    \begin{equation} 
    \mathcal{P}_\alpha \mathcal{E}|U = \mathcal{P}_{\alpha-1} \mathcal{E}|U.
    \label{definition of filtered sheaves}
    \end{equation}
  \end{itemize}
  \label{Filtered Bundles}
\end{definition}

\noindent These $\mathcal{P}_\alpha$'s are essentially sheaves of meromorphic sections with various order of poles. From this principle, we have

\begin{definition}
  For $p \notin D$, we define the order of $s$ at $p$, or $\nu_p(s)$, to be the ordinary pole order. For $p \in D$, open set $U \subset C$ such that $U \cap D = \{p\}$, and a section $s$ of $\mathcal{E}$ over $U$, we define the {\em order} of $s$ at $p$ to be

  \begin{equation}
    \nu_p(s) := \inf \{\alpha\;|\;s \in \mathcal{P}_\alpha \mathcal{E}\}.
    \label{order of a section}
  \end{equation}
\end{definition}

Toward our specific example, we have

\begin{definition}
A {\em filtered $SL(K)$-Higgs Bundle} on $(C,D)$ is a pair $(\mathcal{E},\varphi)$, where $\mathcal{E}$ is a filtered $SL(K)$ bundle on $(C,D)$, and $\varphi$ is a traceless endomorphism (or section of $\mathcal{E}^* \otimes \mathcal{E}$), holomorphic on $C-D$.
\label{Filtered Higgs Bundles}
\end{definition}

The degree of a filtered bundles, similar to ordinary ones, are defined by counting zeros and poles of global sections.

\begin{definition}
  For a filtered {\em line} bundle $\mathcal{L}$ over $(C,D)$, we define its {\em parabolic degree} by
  \begin{equation}
    p\deg \mathcal{L} := -\sum_{p \in C} \nu_p(s),
    \label{parabolic degree of line bunle}
  \end{equation}
\noindent where $s$ is any global meromorphic section. For a rank $K$ filtered bundle $\mathcal{E}$, we define

\begin{equation}
  p\deg \mathcal{E} := p\deg \wedge^K \mathcal{E}.
  \label{parabolic degree of vector bundle}
\end{equation}

\end{definition}

\noindent It is a standard fact form complex geometry that the sum above is well defined. That is, the order is zero except at finitely many points, and the sum is independent of the choice of sections.

In this article, we are interested in the case for $C = \cp$ and $D = \{\infty\}$, where filtered line bundles naturally generalize the standard holomorphic line bundles on $\cp$.

\begin{definition}
  For each $\alpha \in \mathbb{R}$, we define the filtered line bundle $\mathcal{O}(\alpha)$ over $(\cp,\{\infty\})$ to be the increasing filtration of $\mathcal{O}_{\cp}(*\{\infty\})-$modules with pole order at $\infty $ no greater than $\alpha$, filtered by pole order shifted by $-\alpha$.
\end{definition}

\noindent An important fact, proved in \cite{FN} analogously to \cite{Hu}, is

\begin{lemma}
  A filtered $SL(K)$-bundle $\mathcal{E}$ over $(\cp,\{\infty\})$ can be decomposed into

  \begin{equation}
    \mathcal{E} = \bigoplus_{i=1}^K \mathcal{O}(\alpha_i),
    \label{Decomposition of filtered bundle}
  \end{equation}
\noindent where $\sum \alpha_i =0$.
\end{lemma}

Next, we involve the Higgs fields $\varphi$ and equip the pair $(\mathcal{E},\varphi)$ the prescribed types of singularity at $\infty$. The category of Higgs bundles considered in \cite{FN} are called {\em good filtered Higgs bundles} on $(\cp,\{\infty\})$, denoted by $\mathcal{C}_{K,N}$. These are bundles possessing some diagonalizability in some ramified coordinate around $\infty$. The constructions are somewhat complicated. Interested readers may refer to \cite{FN} or \cite{M} for more details. The explicit way to describe them is the existence of gauges $g$ near $\infty$ such that  

\begin{equation}
  \varphi^g = \varphi_{\text{model}} + \tilde{\varphi},
  \label{good Higgs fields form}
\end{equation}

\noindent where $\tilde{\varphi}$ is holomorphic near $\infty$ and 

\begin{equation}
  \varphi_{\text{model}} =
  \begin{pmatrix}
    e^{\frac{2\pi i}{K}} &  &  &  \\
     & e^{\frac{4\pi i}{K}}&  &  \\
     &  &  \ddots  &  \\
     &   &   &  1
  \end{pmatrix} z^{\frac{N}{K}} dz.
  \label{varphi model}
\end{equation}

\noindent The key property is

\begin{proposition}[Proposition 2.19 in \cite{FN}]
  For every $(\mathcal{E},\varphi) \in \mathcal{C}_{K,N}$, the Higgs field $\varphi$ raises the parabolic degree by $\frac{N}{K}$. Moreover, the decomposition given in \eqref{Decomposition of filtered bundle} satisfies
  \begin{equation}
    \sum_{i=1}^K (\alpha_i-\alpha_{i+1})=0.
    \label{cyclic partition of N}
  \end{equation}
  \label{decomposition from cyclic partition}
\end{proposition}

\noindent Equivalently, if we consider $\mathbf{b} = (b_1,\ldots,b_K)$ by

\begin{equation}
  \mathbf{b_i} := \alpha_i - \alpha_{i+1} + \frac{N}{K},
  \label{definition of partition}
\end{equation}

\noindent then $\mathbf{b}$ forms a $K$-partition of $N$. The association is unique up to a cyclic permutation of $b_i's$. In this way, we classify good filtered bundles by cyclic $K$-partitions of $N$. Let $\mathcal{M}_{K,N}$ be the isomorphism classes of good filtered bundles, we arrive at the stratification

\begin{equation}
  \mathcal{M}_{K,N} = \bigsqcup_{[\mathbf{b}]} \mathcal{M}_{K,N}^{[\mathbf{b}]},
  \label{stratification}
\end{equation}

\noindent where $\mathcal{M}_{K,N}^{[\mathbf{b}]}$ are is the gauge classes of good filtered Higgs bundles determined by the cyclic class represented by the partition $\mathbf{b}$.

Similar to the smooth case, the space $\mathcal{M}_{K,N}$ also fibers over the space of invariant polynomials. The specific conditions imposed by good filtered Higgs bundles require the base space to be of certain forms:

\begin{proposition}[Proposition 2.14 of \cite{FN}]
  The space $\mathcal{M}_{K,N}$ fibers over the space of polynomials $\mathcal{B}_{K,N}$ of the form

  \begin{equation}
    (\lambda^K-z^Ndz^K) + (P_2(z)dz^2\lambda^{K-2}+\cdots+P_i(z)dz^i\lambda^{K-i}+\cdots+ P_K(z)dz^K,
    \label{char. poly. specific form}
  \end{equation}
\noindent where

\begin{equation}
  \deg P_i \leq \frac{N(i-1)}{K} -1.
  \label{degree of coefficient}
\end{equation}
The fibration $\pi: \mathcal{M}_{K,N} \to \mathcal{B}_{K,N}$ sends a pair $(\mathcal{E},\varphi)$ to the characteristic polynomial of $\varphi$.
\label{Hitchin base description}
\end{proposition}

\noindent In view of these preliminary results, our interested space, $\mathcal{M}_{2,3}$, consists of pairs $(\mathcal{E},\varphi)$ with

\begin{equation}
  \text{char} \varphi = \lambda^2 - (z^3+u) dz^2,
  \label{char poly for 2,3}
\end{equation}

\noindent where $u \in \mathbb{C}$. Since $N=3$ has precisely two 2-partition up to cyclic permutation, namely $(3,0)$ and $(2,1)$, the moduli space is decomposed into two strata. From  Proposition \ref{decomposition from cyclic partition} and \ref{Hitchin base description} and standard facts from linear algebra, one may readily verify

\begin{proposition}[Propositions 6.2 and 6.3 in \cite{FN}]
  $\mathcal{M}_{2,3}^{[(3,0)]} \simeq \mathbb{C}$ and is represented by, for each $u \in \mathbb{C}$, by
  \begin{equation}
    \mathcal{E}_u = \mathcal{O}\left(\frac{3}{4}\right)\oplus \mathcal{O}\left(-\frac{3}{4}\right), \hspace{1cm} \varphi_u = \begin{pmatrix}
                                                                                                                    0 & z^3+u \\
                                                                                                                    1 & 0
                                                                                                                  \end{pmatrix}dz;
                                                                                                                  \label{M small description}
  \end{equation}
  $\mathcal{M}_{2,3}^{[(2,1)]} \simeq \mathbb{C}^2$ and is represented by, for each $(\gamma,\omega) \in \mathbb{C}^2$, by
    \begin{equation}
    \mathcal{E}_{\gamma,\omega} = \mathcal{O}\left(\frac{1}{4}\right)\oplus \mathcal{O} \left(-\frac{1}{4}\right), \hspace{1cm} \varphi_{\omega,\gamma} = \begin{pmatrix}
                                                                                                                    \gamma & \frac{z^3-\omega^3}{z-\omega} \\
                                                                                                                    z-\omega & -\gamma
                                                                                                                  \end{pmatrix}dz.
                                                                                                                  \label{M big description}
  \end{equation}
  \label{Description of M small and big}
\end{proposition}

\noindent The former strata was referred to as $\Ms$ and the latter $\Mb$. For each $u \in \mathbb{C}$. The fibration structure of this moduli space can be described explicitly. For each $u \in \mathbb{C}$, $\pi^{-1}(u)$ consists of all Higgs fields with determinant $(z^3+u)dz^2$. In terms of the two stratum above, it consists of one point from $\Ms$, and points from $\Mb$ such that $\gamma^2-\omega^3 =u$, an affine cubic curve in $\mathbb{C}^2$. In fact, by considering the gauge

\begin{equation}
  g =i \begin{pmatrix}
        \frac{\gamma}{\omega} & \omega +z \\
        0 & -\frac{\gamma}{\omega},
      \end{pmatrix}
      \label{gauge to go from Mbig to Msmall}
\end{equation}

\noindent authors in \cite{FN} pointed out that $\varphi_u$ is the limit of $\varphi_{\omega,\gamma}$ as $(\omega,\gamma) \to \infty$ in ways compatible to the filtration at $\infty$. The fibers are smooth except at $u=0$, known as the {\em central fiber}. This fiber contains two fixed points of certain $\mathbb{C}^*$ action. We will not focus on this aspect. Rather, we aim to construct harmonic metrics on the whole of $\mathcal{M}_{2,3}$ and generalize the regulated $L^2$ norm, that are only defined for $u=0$ in \cite{FN}, in smooth ways.

\section{Harmonic Metrics on $\Ms$}
Given a good filtered Higgs bundle $(\mathcal{E},\varphi)$ over $(C,D)$, recall that a {\em harmonic metric} on $(\mathcal{E},\varphi)$ is a Hermitian metric $h$ on $\mathcal{E}$, such that

\begin{equation}
  F_h + [\varphi,\varphi^{*_h}]=0.
  \label{Hitchin equation 2}
\end{equation}

\noindent The existence of such metrics are guaranteed:

\begin{theorem}[\cite{BP},\cite{M}]
  Harmonic metric on good filtered bundle exists, unique up to scalar multiple.
  \label{Existence of Harmonic Metric}
\end{theorem}

\noindent Here, $F_h$ is the curvature of the Chern connection associated to the metric $h$. Based on the existence theorem and motivated by Proposition 3.10 of \cite{FN}, we construct partial differential equations associated with \eqref{Hitchin equation 2} for $\Ms$.

Given $u \in \mathbb{C}$ corresponding to the Higgs pair
  \begin{equation}
    \mathcal{E}_u = \mathcal{O}\left(\frac{3}{4}\right)\oplus \mathcal{O}\left(-\frac{3}{4}\right), \hspace{1cm} \varphi_u = \begin{pmatrix}
                                                                                                                    0 & z^3+u \\
                                                                                                                    1 & 0
                                                                                                                  \end{pmatrix}dz,
                                                                                                                  \label{M small description 2}
  \end{equation}

\noindent we consider the existed harmonic metric of the form

\begin{equation}
  h_u = \begin{pmatrix}
          |z^3+u|^{\frac{1}{2}} e^{\psi_u} & 0 \\
          0 & |z^3+u|^{-\frac{1}{2}} e^{-\psi_u}
        \end{pmatrix}.
        \label{harmonic metric on Msmall}
\end{equation}

\noindent for some real valued function $\psi_u$ defined on

 \[\mathbb{C}_u := \mathbb{C} - \{-u^{\frac{1}{3}}\}\]

\noindent Proposition 3.10 and Lemma 3.13 of \cite{FN} can be modified without major difficulty.

\begin{theorem}
  For each $u \in \mathbb{C}$, the functions $\psi_u$ in \eqref{harmonic metric on Msmall} satisfies the elliptic equation
  \begin{equation}
    \Delta \psi_u = 2|z^3+u| \sinh (2\psi_u),
    \label{elliptic equation for Msmall}
  \end{equation}

  \noindent with the following boundary asymptotic behaviors
  \begin{itemize}
    \item $\psi_u \sim \log |z^3+u|$ as $z \to -u^{\frac{1}{3}}$
    \item $\psi_u \to 0$ exponentially $|z| \to \infty$.
  \end{itemize}
  Here, $\Delta = 4 \frac{\p^2}{\p z \p \bar{z}}$ is the Laplacian operator on $\mathbb{C}$.
  \label{individual solution of each u}
\end{theorem}

\begin{proof}
  With respect to $h_u$ in \eqref{harmonic metric on Msmall}, the adjoint of $\varphi_u$ can be readily computed
  \begin{equation}
    \varphi_u^{*_h}=
    \begin{pmatrix}
      0 & |z^3+u|e^{2\psi_u} \\
      \overline{(z^3+u)}|z^3+u|^{-1}e^{-2\psi_u} & 0
    \end{pmatrix}.
    \label{adjoint Higgs field small}
  \end{equation}

  \noindent Equation \eqref{elliptic equation for Msmall} can then be obtained from straightforward computations. The asymptotic behavior near $-u^{\frac{1}{3}}$ follow from the a priori smoothness of $h$ at the entire $\mathbb{C}$ on $\cp$ guaranteed by Theorem \ref{Existence of Harmonic Metric}.

  To show the decay at $\infty$, we consider

  \begin{eqnarray}
    \frac{1}{2} \Delta |\psi_u|^2 &=& \psi_u \Delta \psi_u + \norm{\nabla \psi_u}^2 \nonumber\\
                                 &=& 2 |z^3+u| \psi_u \sinh(2\psi_u) + \norm{\nabla \psi_u}^2 \nonumber\\
                                 &\geq& 2|z^3+u| \psi_u^2 \nonumber\\
                                 &\geq& 0. \label{energy estimate}
  \end{eqnarray}

\noindent The second equality uses \eqref{elliptic equation for Msmall} and the first inequality follows from elementary calculus. By maximum principle, $|\psi_u|^2$ does not attain a maximum on $\mathbb{C}_u$. Since $|\psi_u|^2$ is positive and bounded away form the only singularity $-u^{\frac{1}{3}}$, by the global smoothness of $h_u$, we have

\begin{equation}
  \lim_{|z|\to \infty} \Delta |\psi_u|^2 =0
  \label{decay of Laplacian of norm 2}
\end{equation}

\noindent due to mean value inequality. Equation \eqref{energy estimate} then implies that
\begin{equation}
  \lim_{|z|\to \infty} |\psi_u|^2 =\lim_{|z|\to \infty} \psi_u=0.
  \label{decay of psi u}
\end{equation}

For $|z|$ large enough so that $4|z^3+u| \geq |u|^4$, (more precisely, $|z| \geq \left(|u|+\frac{|u|^4}{4}\right)^{\frac{1}{3}}$), \eqref{energy estimate} yields

\begin{equation}
  \left(\Delta - |u|^4\right) |\psi_u|^2 \geq 0.
  \label{superharmonic equation for psi u}
\end{equation}

\noindent Therefore, $|\psi_u|^2$ is bounded by the bounded solution to
\begin{equation}
  \left(\Delta - |u|^4\right) F_u = 0,
  \label{harmonic equation for F u}
\end{equation}

\noindent which has exponential decay at large $|z|$. In fact, since an explicit bounded $F_u$ is

\begin{equation}
  F_u =  e^{-|u|^2 z - |u|^2 \bar{z}},
  \label{explicit harmonic equation}
\end{equation}

\noindent we conclude that every $\psi_u$ decays exponentially with decay rate of $|u|$ at large $|z|$.

\end{proof}

\noindent Note that the Toda equations considered in \cite{FN} are precisely \eqref{elliptic equation for Msmall} with $u=0$. There, in addition, the fixed point assumption allows the gauge function $\psi_u$ to be radial. We do not assume such property here due to the presence of $u$. Instead, we consider \eqref{elliptic equation for Msmall} as a smooth perturbation of the Toda equation and study the dependence of solutions $\psi_u$ on $u$. The exponential decay holds at $u=0$ as well, as proved in \cite{FN}.

Our main result is to show that $\psi_u \to \psi$, where $\psi$ is the solution to \eqref{elliptic equation for Msmall} for $u=0$:

\begin{equation}
  \Delta \psi = 2 |z^3| \sinh(2\psi)
  \label{Toda like equation}
\end{equation}

\noindent as $u \to 0$ in some appropriate sense. To prove it, we need more uniform controls of $\psi_u$.

\begin{lemma}
  For every compact set $\Omega$ of $\mathbb{C}^*$, the sets
  \[\{\norm{\psi_u}_{L^\infty(\Omega)}\}_u\hspace{0.3cm} \text{and} \hspace{0.5cm} \{\norm{\nabla \psi_u}_{L^\infty(\Omega)}\}_u\]
  \noindent are both bounded.
  \label{boundedness of norms}
\end{lemma}

\begin{proof}
We first bound $\psi_u's$ uniformly. This follows from the exponential decay of each $\psi_u$ for large enough $|z|$. Since $\Omega$ is a compact subset of $\mathbb{C}^*$,

\begin{equation}
  \delta := \inf\{\left(|w|+\frac{|w|^4}{4}\right)^{\frac{1}{3}}\;|\; w \in \Omega\} > 0.
  \label{positive infimum}
\end{equation}

\noindent From the proof of Theorem \ref{individual solution of each u}, it follows that all the $\psi_u$ are bounded by
\[e^{-\delta z + \delta \bar{z}},\]
 \noindent which is a uniformly bounded function on $\Omega$.

The uniform bound on $\psi_u$'s therefore implies the uniform bound on $\Delta \psi_u$'s. We use these to prove the uniform bound on $\norm{\nabla \psi_u}$'s. Pick a bump function $\eta$ supported on $\Omega$. Using integration by parts and H\"older inequality on the integral,

\[\int_\Omega \eta^2 \psi_u \Delta \psi_u,\]

\noindent we have

\begin{eqnarray}
  0 &\leq& \int_\Omega \norm{\nabla \psi_u}^2 \nonumber \\
    &\leq& 2 \left(\int_\Omega (\psi_u)^2 \norm{\nabla \eta}^2 \right)^{\frac{1}{2}} - \int_\Omega \eta^2 \psi_u\Delta \psi_u.
    \label{estimate L 2 norm of nabla psi u}
\end{eqnarray}

\noindent The last integral is bounded over $u$ due to previous constructions, and therefore

\[\int_\Omega \norm{\nabla \psi_u}^2\]

\noindent are bounded over $u$.

Finally, we suppose, toward a contradiction, that $\norm{\nabla \psi_u}$ are unbounded in $L^\infty(\Omega)$. Therefore, for each $u$, there exists $z_u \in \Omega$ such that the set $\{\norm{\nabla \psi_u}(z_u)\}$ are unbounded. However, consider the mean value inequality

\begin{equation}
  \norm{\nabla \psi_u(z)-\nabla \psi_u (z_u)} \leq C\norm{\Delta \psi_u}_{L^\infty(\Omega)},
  \label{mean value inequality}
\end{equation}
\noindent where $C$ is the diameter of $\Omega$. We have, for every $u$,

\begin{equation}
  \int_\Omega \norm{\nabla \psi_u}^2 \geq \int_\Omega \left(\norm{\nabla \psi_u}(z_u)-C\norm{\Delta \psi_u}_{L^\infty(\Omega)}\right)^2.
  \label{final comparison}
\end{equation}

\noindent We have arrived at a contradiction, since the left hand side is bounded while the right hand side is not.

\end{proof}

With this technical lemma, we now prove the main theorem for $\Ms$.

\begin{theorem}
  The family of solutions $\psi_u$ to \eqref{elliptic equation for Msmall} has a sequence that converges to $\psi$ in \eqref{Toda like equation}, in $C^2$ on $\mathbb{C}^*$, as $u \to 0$.
\end{theorem}

\begin{proof}
  We first clear some subtlety. The convergence here is pointwise on $\mathbb{C}^*$. However, in case $z \in \mathbb{C}^*$ is a singularity for some $\psi_v$, namely $z=-v^{\frac{1}{3}}$, we then consider $\psi_u$'s with $|u| \leq \frac{|v|^{\frac{1}{3}}}{2}$. Such a family will not be singular there and the pointwise convergence is therefore well defined as $u \to 0$. The sequential convergence starts from some $u_1$ with $u_1 < \frac{|v|^{\frac{1}{3}}}{2}$. 

  Next, we show that on any compact set $\Omega$ on $\mathbb{C}^*$, $\{\psi_u\}$ has a sequence that converges, in $C^2(\Omega)$, to a $C^2$ function $\psi_0$.  This follows from Schauder's estimate on the elliptic equation \eqref{elliptic equation for Msmall}:

  \begin{equation}
    \norm{\psi_u}_{C^{2,1}(\Omega)} \leq C_\Omega \left(\norm{\psi}_{L^\infty(\Omega)} + \norm{2|z^3+u|\sinh(2\psi_u)}_{C^{0,1}(\Omega)}\right).
    \label{Schauder estimate M small}
  \end{equation}

  The first term on the right is uniformly bounded from Lemma \ref{boundedness of norms}. The second term is dominated by $\norm{z^3+u}_{C^{0,1}}(\Omega)$, which is clearly bounded on compact $\Omega$, and $\norm{\sinh(2\psi_u)}_{C^{0,1}(\Omega)}$. To bound the latter, we note that for each $z,w \in \Omega$,

  \begin{equation}
    \left|\frac{\sinh(2\psi_u(z))-\sinh(2\psi_u(w))}{z-w}\right| \leq 2 \cosh(\psi_u) \norm{\nabla \psi_u}_{L^\infty(\Omega)}
    \label{mean value inequality}
  \end{equation}

  \noindent by mean value inequality. Here $\psi_u$ is between $\psi_u(z)$ and $\psi_u(w)$, and is therefore uniformly bounded. By the theorem of Arzela-Ascoli, we have a sequence $\{\psi_{u_j}\} \subset \{\psi_u\}$ with $u_j \to 0$ as $\infty$ that converges to a $C^2$ function, $\psi_0$, on $\Omega$.

  It remains to show that $\psi_0 = \psi$ on $\Omega$. We first note that $\psi_0$ satisfies the same \eqref{Toda like equation} for $\psi$. Indeed, for each $j$, we have

  \begin{eqnarray}
    \norm{\Delta \psi_0 - 2|z^3|\sinh(2\psi_0)}_{L^\infty(\Omega)} &\leq& \norm{\Delta \psi_0 - \Delta \psi_{u_j}}_{L^\infty(\Omega)} \nonumber \\
                                                             &\;& +\norm{\Delta \psi_{u_j}-2|z^3 + u_j| \sinh(2\psi_{u_j})}_{L^\infty(\Omega)} \nonumber \\
                                                             &\;& +\norm{2|z^3+u_j|\left(\sinh(2\psi_{u_j})-\sinh(2\psi_0)\right)}_{L^\infty(\Omega)} \nonumber \\
                                                             &\;& + \norm{2\left(|z^3 + u_j|-|z^3|\right)\sinh(2\psi_0)}_{L^\infty(\Omega)}. \label{triangle inequalities}
                                                             \end{eqnarray}

  \noindent Each of the term on the right hand side can be made arbitrarily small from previous facts and constructions. Therefore, $\psi_0$ gives a harmonic metric at $u=0$ as well. The two metrics
  \begin{equation}
  h_0 = \begin{pmatrix}
          |z^3|^{\frac{1}{2}} e^{\psi_0} & 0 \\
          0 & |z^3|^{-\frac{1}{2}} e^{-\psi_0}
        \end{pmatrix}, \;\;
  h = \begin{pmatrix}
          |z^3|^{\frac{1}{2}} e^{\psi} & 0 \\
          0 & |z^3|^{-\frac{1}{2}} e^{-\psi}
        \end{pmatrix}
\end{equation}

\noindent are constant multiple of each other, by Theorem \ref{Existence of Harmonic Metric}. Therefore, $\psi_0$ and $\psi$ are differed by a constant $K$. Since they both satisfy \eqref{Toda like equation}, we have

\begin{equation}
  \sinh(2\psi_0 + K) = \sinh(2\psi)
  \label{uniqueness equation}
\end{equation}

\noindent This forces $K=0$ since $\sinh$ is injective and the proof is complete.
\end{proof}

\section{Harmonic Metrics on $\Mb$}
The equations for harmonic metrics on $\Mb$  are much less explicit, except for $\gamma=0$, where equations are similar to \eqref{elliptic equation for Msmall}. For general $\gamma$, we define regulated $L^2$ norms in rather implicit ways.

We start with $\gamma =0$, with the given Higgs pair

 \begin{equation}
    \mathcal{E}_{0,\omega} = \mathcal{O}\left(\frac{1}{4}\right)\oplus \mathcal{O} \left(-\frac{1}{4}\right), \hspace{1cm} \varphi_{\omega,0} = \begin{pmatrix}
                                                                                                                    0 & P_0 \\
                                                                                                                    Q_0 & 0
                                                                                                                  \end{pmatrix}dz.
                                                                                                                  \label{M big at gamma =0}
  \end{equation}

\noindent Here,

\begin{equation}
  Q_0 =z-\omega=z+u^{\frac{1}{3}}, \hspace{1cm} P_0= \frac{z^3-\omega^3}{z-\omega}=\frac{z^3+u}{z+u^{\frac{1}{3}}}.
  \label{off diagonal terms}
\end{equation}

\noindent We remind that $\omega$ and $\gamma$ satisfies the equation $\gamma^2-\omega^3=u$ on the entire fiber of $\Mb$ over $u$. With these notations, we may consider harmonic metrics of the form

\begin{equation}
  h_{\omega,0} = \begin{pmatrix}
                   \left|\frac{P_0}{Q_0}\right|^{\frac{1}{2}} e^{\psi_0} & 0 \\
                   0 & \left|\frac{Q_0}{P_0}\right|^{\frac{1}{2}} e^{-\psi_0}.
                 \end{pmatrix}
  \label{harmonic matric for gamma=0}
\end{equation}

The Hitchin equation then turns into the elliptic PDE

\begin{equation}
  \Delta \psi_0 = 2|z^3 +u| \sinh(2\psi_0)
  \label{elliptic equation for gamma=0}
\end{equation}

\noindent with the asymptotic boundary condition near $\omega=-u^{\frac{1}{3}}$:

\begin{equation}
  \psi_0 \sim -\frac{1}{2} \log \left(\frac{z^2+z\omega+\omega^2}{z-\omega}\right).
  \label{boundary condition for gamma=0}
\end{equation}

\noindent Identical arguments as in $\Ms$ imply that $\psi_0$ decays exponentially as $|z| \to \infty$.

For $\gamma \neq 0$, we have

\begin{equation}
  \varphi_\gamma = \begin{pmatrix}
                              \gamma & P_\gamma \\
                              Q_\gamma & -\gamma
                            \end{pmatrix},
                            \label{Higgs field gamma}
\end{equation}

\noindent where

\begin{equation}
  Q_\gamma =z-\omega=z+(u-\gamma^2)^{\frac{1}{3}}, \hspace{1cm} P_\gamma= \frac{z^3-\omega^3}{z-\omega}=\frac{z^3+u-\gamma^2}{z+(u-\gamma^2)^{\frac{1}{3}}}.
  \label{off diagonal terms}
\end{equation}

\noindent It is not possible to turn Hitchin equations into a scalar PDE here. However, the smooth harmonic metric $h_{\omega,\gamma}$ satisfying

\begin{equation}
  F_{h_\gamma}+[\varphi_\gamma,\varphi_\gamma^{*_{h_\gamma}}]=0
  \label{General Hitchin Equtaion on Mbig}
\end{equation}

\noindent still exists due to Theorem \ref{Existence of Harmonic Metric}.

Viewing Hermitian structures as smooth sections of bundles of $SU(2)$-endomorphisms with pointwise inner product

\begin{equation}
  <h_1,h_2> := Tr (h_1h_2^*),
  \label{inner product on matrices}
\end{equation}

\noindent we expect harmonic metrics to depend smoothly on the parameter $\gamma$. In fact, we have

\begin{proposition}
The harmonic metrics $h_\gamma$ approach $h_0$ as $\gamma \to 0$ in any Sobolev norms.
\label{convergence of harmonic metrics}
\end{proposition}

\begin{proof}
  Fix a gauge so that all Hermitian structure is of determinant 1 and harmonic metric for each $\gamma$ is unique. We consider the map

  \begin{equation}
    \mathcal{H}: A^{1,0}(\mathfrak{sl}(E) \oplus A^0(SU(E)) \to A^2(\mathfrak{su}(E)) \oplus A^{1,1}(\mathfrak{sl}(E))
  \end{equation}

  \noindent defined by

  \begin{equation}
    \mathcal{H}(\varphi,h) := F_{h} + [\varphi,\;\varphi^{*_h}],
    \label{right side of Hitchin}
  \end{equation}

  \noindent where $F_h$ is the curvature uniquely determined by $h$. Clearly, $\mathcal{H}(\varphi_\gamma,h_\gamma)=0$ for all $\gamma \in \mathbb{C}$. For each $\gamma$, we have

  \begin{equation}
    h_\gamma = h_0 e^{\eta_\gamma}
    \label{self adjoint infinitesimal gauge}
  \end{equation}

  \noindent for some $\mathfrak{su}(E)$-valued section $\eta_\gamma$. Plugging into \eqref{right side of Hitchin}, we obtain

  \begin{eqnarray}
    \mathcal{H}((\varphi_\gamma,h_\gamma)) &=& F_{h_\gamma} + [\varphi_\gamma\;,\;\varphi_\gamma^{*_{h_\gamma}}] \nonumber \\
    &=& F_{h_0} + \p_{h_0}\bar{\p}_{h_0} \eta_\gamma + o(|\eta_\gamma|^2)+[\varphi_\gamma,\;\varphi_\gamma^{*_{h_\gamma}}]  \nonumber \\
    &=&  \p_{h_0}\bar{\p}_{h_0} \eta_\gamma + [\varphi_\gamma \; ,\; \varphi_\gamma^{*_{h_\gamma}}] - [\varphi_0 \; , \; \varphi_0^{*_{h_0}}] + o(|\eta_\gamma|^2)\nonumber \\
    &=&0.
    \label{elliptic equation for eta}
  \end{eqnarray}

\noindent The infinitesimal variation $\eta_\gamma$ satisfies an elliptic PDE and therefore smoothly depend on the coefficients. Therefore, entries of $h_{\omega,\gamma}$ depend smoothly on $\gamma$ since $\varphi_\gamma$'s do. Moreover, every harmonic metric is smooth on $\mathbb{CP}^1$. For $\gamma$'s on a bounded set, $\|h_\gamma\|_{L^\infty}$ are uniformly bounded. Since the harmonic metrics are adapted to the filtration at $\infty$, $\varphi_\gamma$'s are uniformly bounded in $L^\infty$. Moreover, $\varphi_\gamma \to \varphi_0$ smoothly from their definitions.  Elliptic regularities, along with bootstrapping \eqref{elliptic equation for eta}, then show that $\eta_\gamma \to 0$ in all Sobolev norms, and the proof is complete.
\end{proof}

\section{The regulated $L^2$ norm}
We now present our partial generalizations of the regulated $L^2$ norms on $\mathcal{M}_{2,3}$. The generalization is quite natural on $\Ms$. However, the generalization we have defined on $\Mb$ is still somewhat vague. Nevertheless, the regulated $L^2$-norm we have defined on the whole moduli space depend continuously on the parameter $u,\gamma,$ and $\omega$.

In \cite{FN}, the authors defined the formula

\begin{equation}
  \mu([\mathcal{E},\varphi]) := \frac{i}{\pi} \int Tr \left(\varphi \wedge \varphi^{*_h}- Id |z|^3\right)\;dzd\bar{z}
  \label{regulated L 2 norm on fixed point}
\end{equation}

\noindent and explicitly compute them on the two fixed points over the central fiber $u=0$.

We note that the second term in the integrand is placed to remove the non-integrable terms on the trace. The integral is still divergent away from the central fiber. On $\Ms$, we define, over the fiber $u \in \mathbb{C}$,

\begin{equation}
    \mu([\mathcal{E}_u,\varphi_u]) := \frac{i}{\pi} \int Tr \left(\varphi_u \wedge \varphi_u^{*_h}- Id |z^3+u|\right)\;dzd\bar{z}.
  \label{regulated L 2 norm on M small}
\end{equation}

\noindent Explicit computations of this integral is difficult. It is, however, convergent and smooth for every $u \in \mathbb{C}$.

\begin{theorem}
  The integral \eqref{regulated L 2 norm on M small} is convergent for every $u \in \mathbb{C}$. Moreover, it is smooth in $u$.
\end{theorem}

\begin{proof}

Since the integrand is clearly smooth in $u$, it is sufficient to show the convergence of the integrals. We re-write

\begin{equation}
  \mu([\mathcal{E},\varphi]) := \frac{i}{\pi} = \int 4\sinh^2(\psi_u)|z^3+u|\;dzd\bar{z}.
  \label{regulated L 2 norm on M small alternative}
\end{equation}

The exponential decay of $\psi_u$ dominates $|z^3+u|$ at large $|z|$ and the remaining concern is the singularities of $\psi_u$ at $z^3=-u$. Near those points, boundary conditions stated in \ref{individual solution of each u} provide asymptotic behavior

\begin{equation}
  \sinh^2(\psi_u) \sim \left(|z^3+u| - \frac{1}{|z^3+u|}\right)^2.
  \label{asymptotic behaviors}
\end{equation}

\noindent The only singular term for small $|z|$ is then $\frac{1}{|z^3+u|}$. For $u \neq 0$, $z^3+u$ has three distinct roots, and therefore the singular term consists of three simple poles. Two dimensional integrations over domain near simple poles are finite, and therefore the integral \eqref{regulated L 2 norm on M small alternative} converges.
\end{proof}

Strictly speaking, the regulated $L^2$-norms defined above and in \cite{FN} are not ordinary norms. The definition, for example, is not homogeneous. They are however still closely associated to $\mathbb{S}^1$-moment maps descriptions from \cite{H}. These regulated $L^2$-norms are more precisely norms on the smooth part of the Higgs fields. The adjusted terms in the integrals are defined to remove the singular parts. In fact, for $\mu$ defined on $\Ms$ in \eqref{regulated L 2 norm on M small}, consider the trivialization away from $z=-u^{\frac{1}{3}}$ given by the gauge

\begin{equation}
  g= -\frac{1}{2} \begin{pmatrix}
                    (z^3+u)^{-\frac{1}{2}} & -1 \\
                    -1 & -(z^3+u)^{\frac{1}{2}}.
                  \end{pmatrix}
                  \label{trivializaing gauge for phi model}
\end{equation}

\noindent This gauge is in fact the required gauge for $[\mathcal{E},\varphi_u]$ to be a {\em good pair} as described in \cite{FN}. In this frame, the integrand of regulated norm \eqref{regulated L 2 norm on M small} vanishes. Therefore,

\begin{equation}
    \mu([\mathcal{E}_u,\varphi_u]) := \frac{i}{\pi} \int_{D_u} Tr \left(\varphi_u \wedge \varphi_u^{*_h}- Id |z^3+u|\right)\;dzd\bar{z},
  \label{regulated L 2 norm on M small near singularity}
\end{equation}

\noindent where $D_u$ is a small disc around $-u^{\frac{1}{3}}$. Near the zero of determinant, we observe that

\begin{equation}
  Tr \left(\varphi_u \wedge \varphi_u^{*_h}- Id |z^3+u|\right) = Tr \left(\hat{\varphi}_u \wedge \hat{\varphi}_u^{*_h}\right),
  \label{smooth part of Higgs field}
\end{equation}

\noindent where

\begin{equation}
  \hat{\varphi}_u = \begin{pmatrix}
                        0 & (z^3+u)(e^{2\psi_u}-1)^{\frac{1}{2}}e^{\psi_u} \\
                        (e^{-2\psi_u}-1)^{-\frac{1}{2}}e^{-\psi_u} & 0
                      \end{pmatrix}.
                      \label{Higgs field near zero}
\end{equation}

\noindent This is a smooth Higgs field on $D_u$ due to the boundary conditions of $\psi_u$ stated in Theorem \ref{individual solution of each u}. Let $\tilde{\varphi}_u$ be the smooth Higgs field formed by patching $\hat{\varphi}_u$ and $\varphi_u$ with appropriate partition of unity, we conclude that the regulated $L^2$ norm on $\Ms$ is

\begin{equation}
  \mu([\mathcal{E}_u,\varphi_u]) := \frac{i}{\pi} \int Tr\left(\tilde{\varphi}_u \wedge \tilde{\varphi}_u^{*_h}\right)\;dzd\bar{z}
  \label{L 2 norm as moment map}
\end{equation}

\noindent with the usual moment map 

\begin{equation}
\mathcal{M}(\tilde{\varphi}) =[\tilde{\varphi}_u,\tilde{\varphi}_u^{*_h}].
\label{moment map}
\end{equation}

The generalization to $\Mb$ is much trickier except at $\gamma=0$. There, the harmonic equation is identical and \eqref{regulated L 2 norm on M small} continues to be defined. For $\gamma \neq 0$, the harmonic metric is no longer diagonal, and it is impossible to turn Hitchin equation into a scalar PDE. We here provide a rather implicit potential formula.

For

\begin{equation}
  \varphi_\gamma = \begin{pmatrix}
                              \gamma & P_\gamma \\
                              Q_\gamma & -\gamma
                            \end{pmatrix},
                            \label{Higgs field gamma again}
\end{equation}

\noindent above, we write the corresponding harmonic metric as

\begin{equation}
  h_\gamma =
  \begin{pmatrix}
    f_{1,\gamma} & g_\gamma \\
    \bar{g}_\gamma & f_{2,\gamma},
  \end{pmatrix}
  \label{harmonic metric at nonzero gamma}
\end{equation}

\noindent with the gauge fixing condition that $\det h_\gamma =1$.

The regulated $L^2$ norm we define on $\Mb$ is

\begin{eqnarray}
      \mu([\mathcal{E}_\gamma,\varphi_\gamma]) &:=& \frac{i}{\pi} \int [Tr \left(\varphi_\gamma \wedge \varphi_u^{*_{h_\gamma}}- 2|P_\gamma Q_\gamma | Id\right) -E_\gamma ]\;dzd\bar{z} \nonumber \\
      &:=& \frac{i}{\pi} \int [Tr \left(\varphi_\gamma \wedge \varphi_u^{*_{h_\gamma}}- 2|z^3+u-\gamma^2 | Id\right) -E_\gamma ]\;dzd\bar{z}.
      \label{regulated L 2 metric on M big}
\end{eqnarray}

\noindent The error term $E_\gamma$ serves similar role to those norms previously defined: to remove the divergent contribution for integrals. After direct computations, it is

\begin{eqnarray}
  E_\gamma = &\;& 2|\gamma|^2 f_{1,\gamma}f_{2,\gamma} + |P_\gamma|^2 f_{2,\gamma}^2 + |Q_\gamma|^2 f_{1,\gamma}^2 \nonumber \\
             &+& 4 Re(\gamma \bar{P}_\gamma f_{2,\gamma}g_\gamma - \gamma \bar{Q}_\gamma f_{1,\gamma}\bar{g}_\gamma - \bar{P}_\gamma Q_\gamma g_\gamma^2) \nonumber \\
             &-& 2|z^3+u-\gamma^2|\cosh(2\psi_0).
             \label{error term in L 2 norm}
\end{eqnarray}

In another words, we have,

\begin{equation}
      \mu([\mathcal{E}_\gamma,\varphi_\gamma]) := \frac{i}{\pi} \int \left(2|\gamma|^2|g_\gamma|^2 +2 |z^3+u-\gamma^2|\sinh^2(\psi_0)\right)\;dzd\bar{z}.
      \label{regulated L 2 metric on M big alternative}
\end{equation}

The integral clearly agrees with the regulated $L^2$ norm of $[\mathcal{E}_0,\varphi_0]$. With the smooth convergence of harmonic metrics from Proposition \ref{convergence of harmonic metrics}, the regulated $L^2$ norms defined above are smooth in $\gamma$ (and therefore $\omega$ as well) on $\Mb$.

To conclude, we have extended the regulated $L^2$ norms from \cite{FN} to the two ends, $\gamma=0,\infty$, of each fiber. Moreover, these regulated $L^2$ norms are essentially the regular $L^2$ norms on the smooth parts of the Higgs fields by removing the singular part at $\infty$ with $\mathbb{S}^1$ moment maps defined in \cite{H}. We however eagerly hope to improve, or adjust, the definition above so we have more insight information for the geometric structure of the toric fibers of $\mathcal{M}_{2,3}$.

\end{onehalfspacing}

\end{document}